\documentclass[11pt]{amsart}
\usepackage{lmodern}
\usepackage[top=1.5in, left=1.2in, right=1.2in, bottom=1.5in]{geometry}
\usepackage{amssymb}
\usepackage{comment}
\usepackage{amsmath, amscd}
\usepackage{csquotes}
\usepackage{amsthm}
\usepackage[all]{xy}
\usepackage{color}
\usepackage{enumitem}
\usepackage{caption}
\usepackage{wrapfig}
\usepackage{graphicx}
\usepackage{pst-all}
\usepackage{hyperref}
\usepackage{tikz-cd}
\usepackage{booktabs,makecell,siunitx,eqparbox}

\usepackage{palatino}

\newtheorem{theorem}{Theorem}[section]
\newtheorem*{theorem*}{Theorem}
\newtheorem{lemma}[theorem]{Lemma}

\theoremstyle{definition}

\newtheoremstyle{fact}
  {.5em}
  {.5em}
  {\itshape}
  {}
  {\itshape}
  {.}
  {.5em}
  {\thmname{#1}\thmnumber{ #2}\thmnote{ (#3)}}
\theoremstyle{fact}

\newtheorem{innercustomgeneric}{\customgenericname}
\providecommand{\customgenericname}{}
\newcommand{\newcustomtheorem}[2]{%
	\newenvironment{#1}[1]
	{%
		\renewcommand\customgenericname{#2}%
		\renewcommand\theinnercustomgeneric{##1}%
		\innercustomgeneric
	}
	{\endinnercustomgeneric}
}

\newcustomtheorem{customthm}{Theorem}

\begin{document}

\title{Quantitative distortion and the Hausdorff dimension of continued fractions}

\author{Daniel Ingebretson}

\begin{abstract}
We prove a quantitative distortion theorem for iterated function systems that generate sets of continued fractions. 
As a consequence, we obtain upper and lower bounds on the Hausdorff dimension of any set of real or complex continued fractions.
These bounds are solutions to Moran-type equations in the convergents that can be easily implemented in a computer algebra system.
\end{abstract}

\maketitle 

\section{Introduction}
Let $ \mathbb{N} \times \mathbb{Z}i $ be the set of Gaussian integers whose first coordinate is positive, and let $ I \subset \mathbb{N} \times \mathbb{Z}i $.
Each right-infinite word $ \omega = (\omega_1, \omega_2, \ldots) \in I^{\infty} $ corresponds to an infinite complex continued fraction
$$
x = \cfrac{1}{\omega_1+\cfrac{1}{\omega_2+\cfrac{1}{\ddots}}}.
$$
The set $ J_I \subset \mathbb{C} $ of continued fractions formed in this way has zero Lebesgue measure in $ \mathbb{C} $.
For $ I = \mathbb{N} \times \mathbb{Z}i $, a number of authors have estimated its Hausdorff dimension (see \cite{GM}, \cite{MU}, \cite{Pri}). 
Recently Falk and Nussbaum \cite{FN} estimated $ \dim_H(J_I) \simeq 1.855 $ up to some error in the algorithm.

If $ I $ is a proper subset of $ \mathbb{N} $ the continued fractions are real, and $ J_I \subset (0,1) \setminus \mathbb{Q} $.
For finite $ I $, estimates on the Hausdorff dimension of $ J_I $ go back as far as 1941 \cite{Goo}, and continue to the present (see \cite{Cus}, \cite{Bum}, \cite{Hen}, \cite{Hen1}, \cite{Hen2}).
The current state-of-the-art is the Jenkinson-Pollicott algorithm \cite{JP}, which can calculate the dimension of $ J_{\{1,2\}} $ accurate to one hundred decimal places \cite{JP1}.
For infinite $ I $, the Falk-Nussbaum algorithm can be adapted to provide good estimates \cite{CLU}.

Virtually all the literature on dimension theory for continued fractions uses technical methods to approximate the spectrum of certain transfer operators defined on suitable spaces of functions in the complex plane.
The Falk-Nussbaum algorithm uses piecewise polynomial interpolation, and the Jenkinson-Pollicott algorithm uses heavy functional analytic machinery, including the Grothendieck theory of nuclear operators.

In this paper, we give a simple alternative method for estimating the Hausdorff dimension of complex or real continued fractions.
To state the theorem, we first fix some notation.
Let $ I^{\ast} $ be the set of finite words on the alphabet $ I $.
Each $ \omega \in I^{\ast} $ defines a finite continued fraction
\begin{equation}
\label{aomega}
a_{\omega} = \cfrac{1}{\omega_1+\cfrac{1}{\omega_2+\cfrac{1}{\ddots +\cfrac{1}{\omega_n}}}}
\end{equation}
which can be uniquely expressed as a quotient of relatively prime Gaussian integers
\begin{equation}
\label{qomega}
a_{\omega} = \frac{p_{\omega}}{q_{\omega}}.
\end{equation}

\begin{customthm}{1}
	\label{thm1}
	Let $ J_I $ be a set of (real or complex) continued fractions.
	For all $ k \geq 1 $,
	$$
	T_k^- < \dim_H(J_I) < T_k^+,
	$$
	where $ t=T_k^{\pm} $ are the unique solutions to the equations
	$$
	\sum_{\omega \in I^k} q_{\omega}^{-t}(1+a_{\omega})^{\pm 2t} = 1.
	$$
	Furthermore, $ T_k^{\pm} $ are monotonic in $ k $, and as $ k \to \infty $,
	$$ 
	|T_k^{\pm} - \dim_H(J_I)| = O\left(\frac{1}{k}\right). 
	$$
\end{customthm}

Mauldin and Urba\'{n}ski \cite{MU}, \cite{MU1} studied continued fractions as limit sets of iterated function system (IFS) generated by maps $ \phi_b(z) = (b+z)^{-1} $ defined on a closed disk $ D \subset \mathbb{C} $.
For each finite word $ \omega \in I^n $ we have a composition
$$
\phi_{\omega} = \phi_{\omega_1} \circ \cdots \circ \phi_{\omega_n},
$$
and if there exists $ K \geq 1 $ such that for all $ \omega \in I^{\ast} $ and $ z, w \in D $,
$$
K^{-1} \leq \frac{|\phi'_{\omega}(z)|}{|\phi'_{\omega}(w)|} \leq K,
$$
we say the IFS has \textit{bounded distortion}.
Mauldin and Urba\'{n}ski \cite{MU1} showed that $ K=4 $ suffices for the continued fraction IFS.
To prove Theorem \ref{thm1}, we upgrade this to a quantitative version.
First, some terminology from Sullivan \cite{Sul}: if $ \omega = (\omega_1, \ldots, \omega_n) $ is a finite word, we denote its \textit{dual} by
$$
\widetilde{\omega} = (\omega_n, \ldots, \omega_1).
$$

\begin{customthm}{2}[Quantitative distortion]
	\label{thm2}
	Let $ \{ \phi_b : D \rightarrow D \}_{b \in I} $ be a continued fraction IFS indexed by $ I \subset \mathbb{N} \times \mathbb{Z}i $.
	For any $ \omega \in I^{\ast} $ and all $ z,w \in D $,
	$$
	(1+\phi_{\widetilde{\omega}}(0))^{-2} \leq \frac{|\phi'_{\omega}(z)|}{|\phi'_{\omega}(w)|} \leq (1+\phi_{\widetilde{\omega}}(0))^2,
	$$
	and the inequalities are strict unless $ z=0 $ and $ w=1 $ or vice versa.
\end{customthm}

This suggests a method to generalize Theorem \ref{thm1} to other IFS; first establish a quantitative distortion result, and use this to estimate the dimension. 
The tighter the bounds on distortion, the better the dimension estimates.

\subsection{Organization of the paper}
In Section \ref{IFS}, we summarize the IFS theory and its connection to continued fractions.
In Section \ref{quantdist}, we prove Theorem \ref{thm2}.
In Section \ref{Hbounds}, we prove Theorem \ref{thm1} and implement the algorithm in Mathematica to obtain explicit numerical bounds on the dimension of $ J_I $ for various alphabets $ I $.
We also discuss how the convergence depends (for better or for worse) on the choice of $ I $.

\section{Iterated function systems}
\label{IFS}
In this section we will summarize some important definitions and results from the general theory of infinite conformal IFS.
For details, see \cite{MU}, \cite{Mau}.
We will first state the results for general IFS. 
From Section \ref{contfrac} through the end of the paper we will specify to the continued fraction IFS.

Consider a countable family $ \{ \phi_i : X \rightarrow X \}_{i \in I} $ of uniformly contracting injective maps of a compact metric space.
Uniform contraction means that there exists $ 0 < \lambda < 1 $ such that 
$$
d(\phi_i(x), \phi_i(y)) \leq \lambda \: d(x,y)
$$
for all $ i \in I $ and $ x,y \in X $.
In addition, assume the \textit{open set condition}: if $ \mathcal{O} = \text{int}(X) $, then 
$$
\phi_i(\mathcal{O}) \cap \phi_j(\mathcal{O}) = \emptyset
$$
for all $ i \neq j $.
We call such a family an \textit{iterated function system} or IFS.

\subsection{Symbolic coding}
Denote by $ I^n $ the words of length $ n $ on the alphabet $ I $, the finite words by $ I^{\ast} $, and the right-infinite words by $ I^{\infty} $.
If $ \omega = (\omega_1, \omega_2, \ldots) \in I^{\infty} $, for each $ n \geq 1 $ we have a truncation
$$
\omega |_n = (\omega_1, \ldots, \omega_n)
$$
of $ \omega $ to length $ n $.
For each $ \omega \in I^n $ there is a composition map
$$
\phi_{\omega} = \phi_{\omega_1} \circ \cdots \circ \phi_{\omega_n}.
$$
If $ \omega \in I^{\infty} $, by uniform contraction we have
$$
\text{diam} \: \phi_{\omega |_n}(X) \leq \lambda^n \: \text{diam} \: X.
$$ 
From this and the nesting property
$$
\phi_{\omega |_{n+1}}(X) \subset \phi_{\omega |_n}(X),
$$
the \textit{coding map} $ \pi : I^{\infty} \rightarrow X $ given by
$$
\pi(\omega) = \bigcap_{n \geq 1} \phi_{\omega |_n}(X)
$$
is well-defined, and its image $ J = \pi(I^{\infty}) $ is called the \textit{limit set} of the system.

\subsection{Topological pressure and Hausdorff dimension}
\label{toppres}
From now on we assume that $ X \subset \mathbb{C} $ and the maps $ \phi_i $ extend to conformal $ C^{1+\alpha} $ diffeomorphisms of an open neighborhood of $ X $.
We also assume they satisfy the bounded distortion property from the introduction.
Under these assumptions the limit
$$
P(t) = \lim_{n \to \infty} \frac{1}{n} \log \sum_{\omega \in I^n} |\phi'_{\omega}(z)|^t
$$
exists and is independent of the choice of $ z $.
$ P $ is called the \textit{topological pressure} of the IFS, and $ t \mapsto P(t) $ is continuous, strictly decreasing and convex on its domain of finiteness.
The pressure is related to the Hausdorff dimension of the limit set by Mauldin and Urba\'{n}ski's generalization of Bowen's equation:
$$
\dim_H(J) = \inf\{t \geq 0 : P(t)<0 \}.
$$
If there exists a $ t $ such that $ P(t)=0 $, then $ t $ is unique and coincides with $ \dim_H(J) $.
Systems whose pressure has a unique zero are called \textit{regular}.
Infinite systems whose cofinite subsystems are all regular are called \textit{hereditarily regular}.

\subsection{Continued fraction IFS}
\label{contfrac}
Following \cite{GM} \cite{MU}, we now give a geometric description of the continued fraction set $ J_I $ from the introduction.

The inversion map $ z \mapsto z^{-1} $ takes the shifted right half-plane $ H = \{z \in \mathbb{C} : \text{Re}(z) \geq 1 \} $ conformally onto the closed disk $ D $ centered at $ z=\frac{1}{2} $ with radius $ \frac{1}{2} $.
The translates $ \{D+b : b \in \mathbb{N} \times \mathbb{Z}i \} $ lie in $ H $, so the inversion maps them to an infinite family of disks in $ D $.
See Figure \ref{fig1}.

\begin{figure}[h]
	\minipage{0.5\textwidth}
	\includegraphics[width=\linewidth, trim={0cm 0.2cm -0.5cm 0cm}, clip]{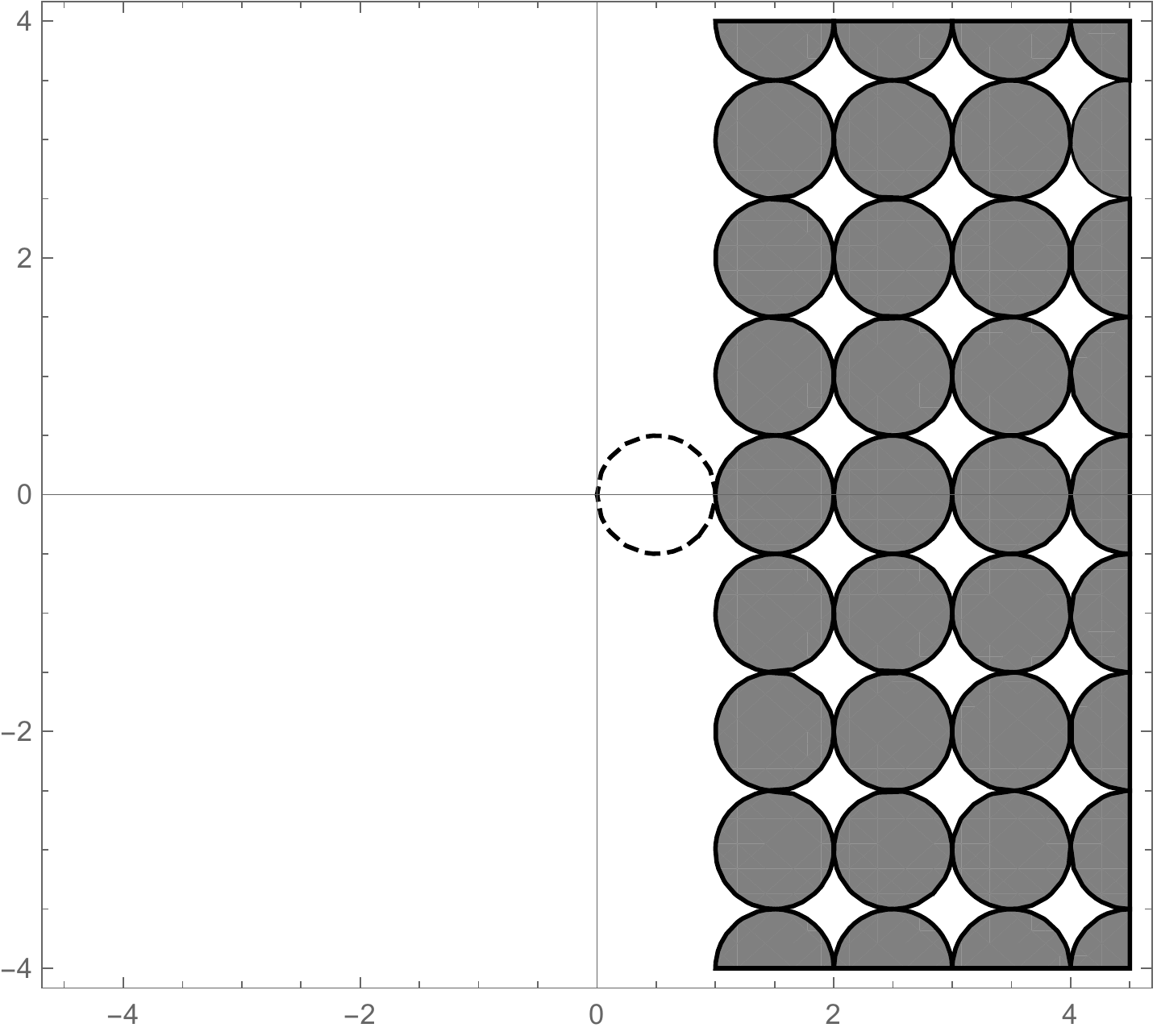}
	\caption*{The translates of $ D $ by $ \mathbb{N} \times \mathbb{Z}i $}
	\endminipage
	\minipage{0.5\textwidth}%
	\includegraphics[scale=0.5, trim={-0.5cm 0.2cm 0cm 0cm}, clip]{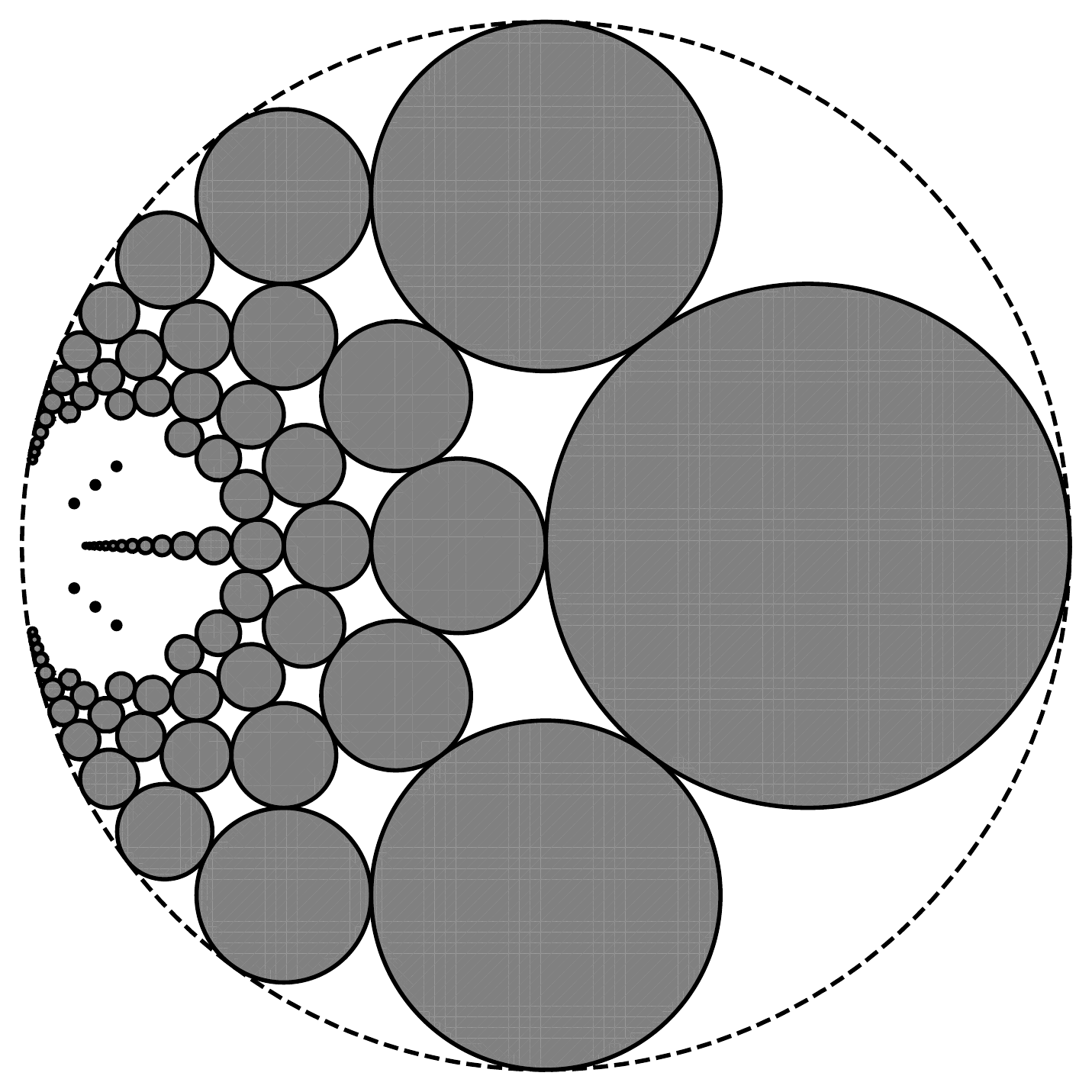}
	\caption*{The images in $ D $ of the translates under the inversion map}
	\endminipage
	\caption{}
	\label{fig1}
\end{figure}

Equivalently, $ \phi_b(D) \subset D $ for each $ b \in I \subset \mathbb{N} \times \mathbb{Z}i $, where $ \phi_b(z) = (b+z)^{-1} $.
These maps satisfy the open set condition, have regularity $ C^{1+\alpha} $, and satisfy bounded distortion.
Thus $ \mathcal{S}_I = \{\phi_b : D \rightarrow D \}_{b \in I} $ is a conformal IFS, with a caveat-- because $ \phi'_1(0) = -1 $ these maps are not uniformly contracting.
This is not a problem though, because (as pointed out in \cite{MU1}) the maps $ \phi_{b_1} \circ \phi_{b_2} $ \textit{are} uniformly contracting and generate the same limit set, which is what we are interested in (see Figure \ref{fig2}).

\begin{figure}[h]
	\includegraphics[width=0.8\linewidth, trim={0cm 0.2cm -0.5cm 0cm}, clip]{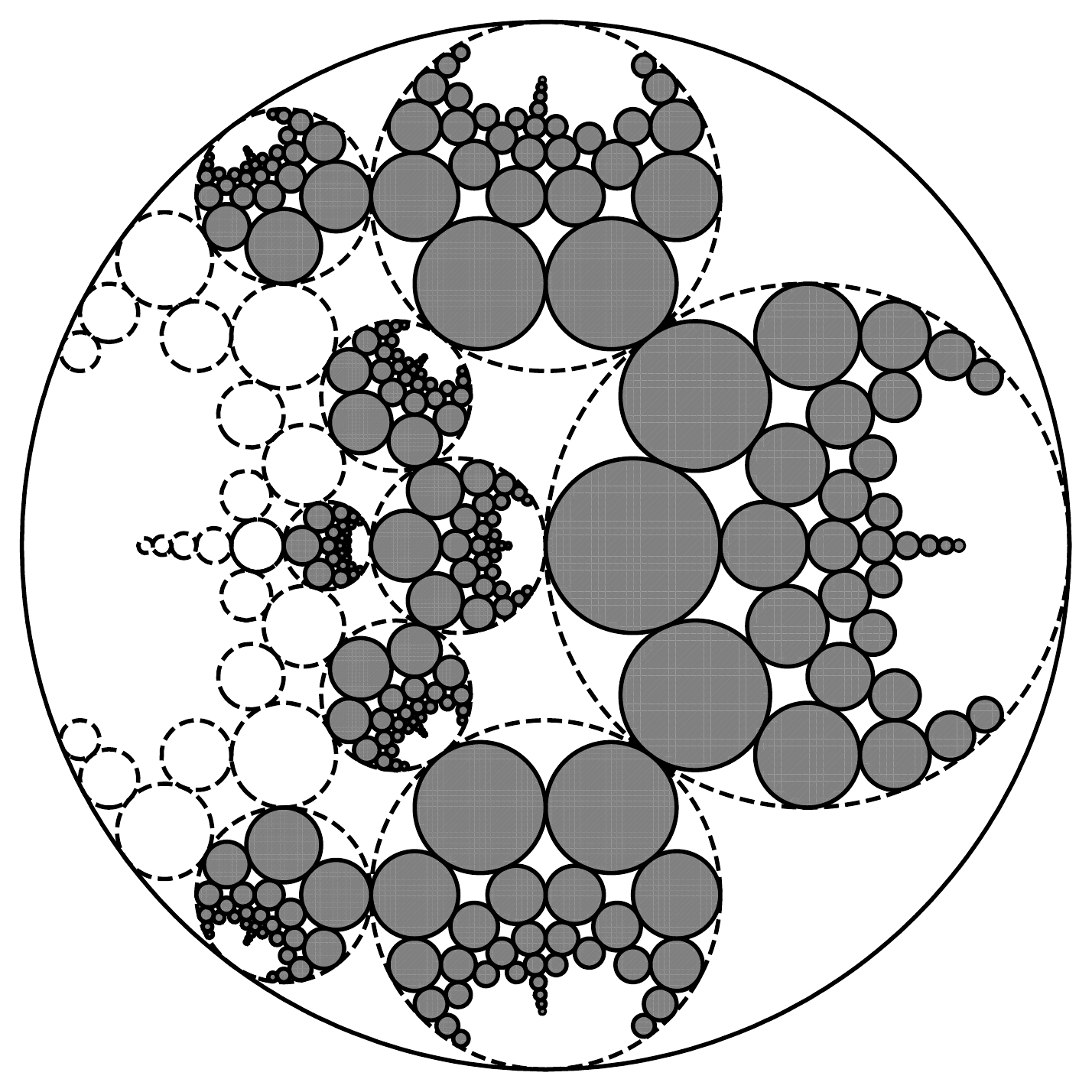}
	\caption{Images of some of the second-order iterates of $ \phi_{b_1} \circ \phi_{b_2}(D) $ for $ b_j \in \mathbb{N} \times \mathbb{Z}i $.}
	\label{fig2}
\end{figure}

The coding map $ \phi : I^{\infty} \rightarrow D $ takes a word to its infinite continued fraction expansion, so the limit set of $ \mathcal{S}_I $ is $ J_I $ discussed in the introduction.
See Figure \ref{fig3} for an illustration where $ \#I = 4 $.

\begin{figure}[h]
	\minipage{0.5\textwidth}
	\includegraphics[width=\linewidth, trim={0cm 0.2cm -0.5cm 0cm}, clip]{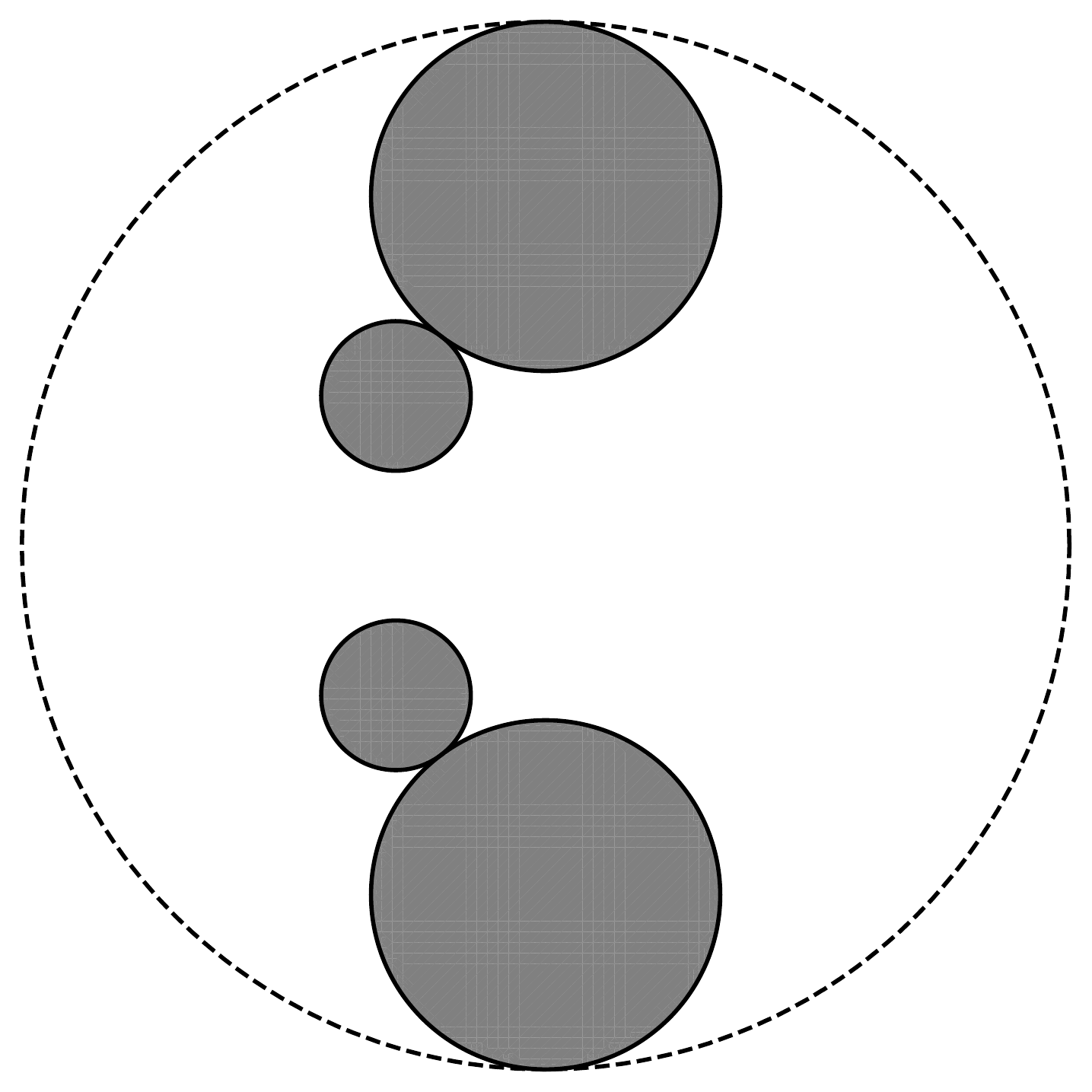}
	\caption*{The first-order iterates $ \phi_{b_1} $}
	\endminipage
	\minipage{0.5\textwidth}%
	\includegraphics[width=\linewidth, trim={-0.5cm 0.2cm 0cm 0cm}, clip]{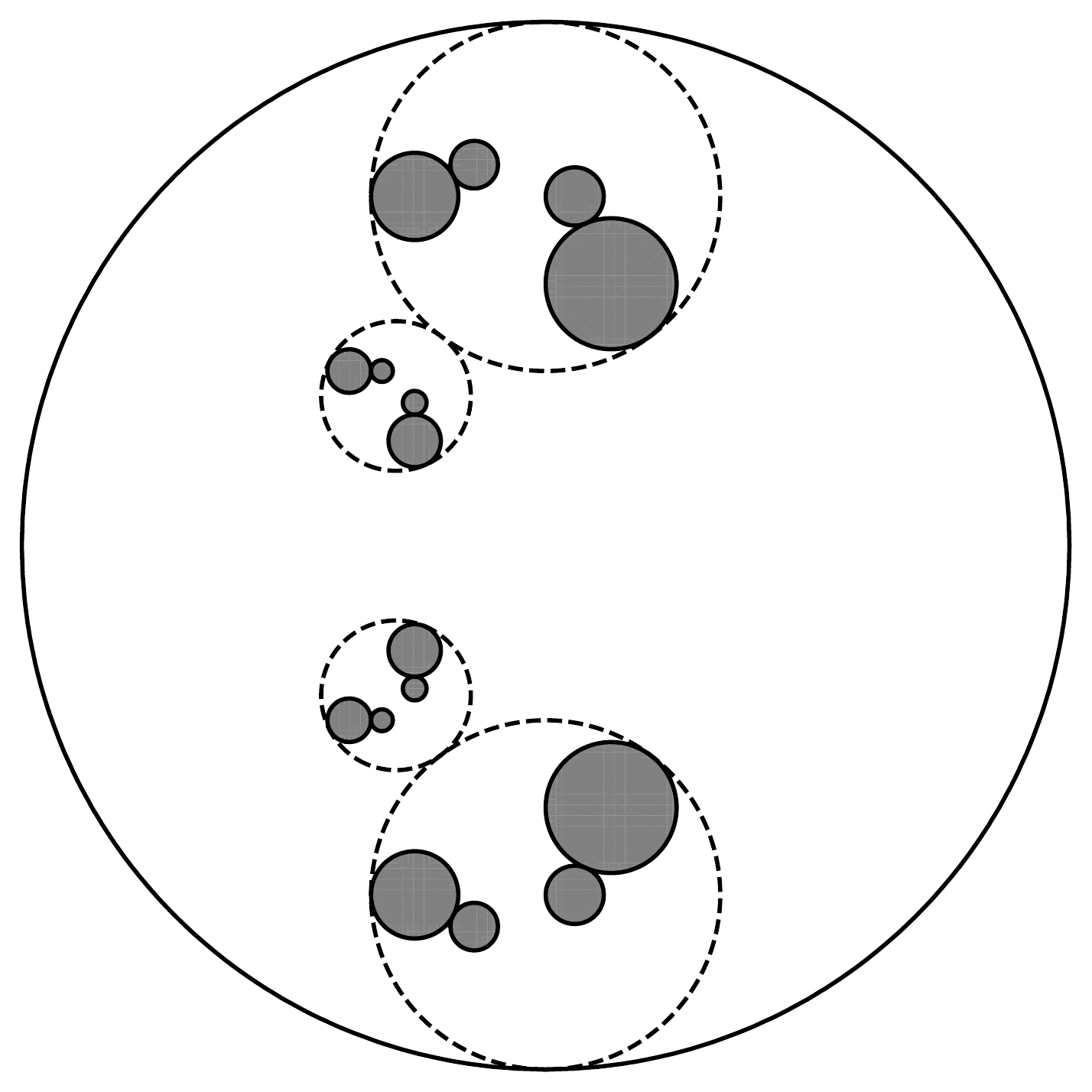}
	\caption*{The second-order iterates $ \phi_{b_1} \circ \phi_{b_2} $}
	\endminipage
	\caption{Images of some iterates of $\phi_b$ where $ b \in \{1 \pm i, 2 \pm i\} $}
	\label{fig3}
\end{figure}

\section{Quantitative distortion for continued fractions}
\label{quantdist}
In this section, we will prove Theorem \ref{thm2}.
First we will introduce some notation and prove some preparatory lemmas.

\subsection{Generalized convergents}
Let $ \omega \in I^{\infty} $.
For each $ n \geq 1 $ and $ z \in \mathbb{C} $, we have a finite continued fraction
$$ 
\phi_{\omega |_n}(z) = \cfrac{1}{\omega_1+\cfrac{1}{\omega_2+\cfrac{1}{\ddots +\cfrac{1}{\omega_{n-1}+\cfrac{1}{\omega_n+z}}}}}.
$$ 
Abusing common notation, we call $ \phi_{\omega |_n}(z) $ the \textit{nth generalized convergent} of $ \pi(\omega) $.
It can be written as a quotient
$$
\phi_{\omega |_n}(z) = \frac{p_{\omega |_n}(z)}{q_{\omega |_n}(z)}.
$$
whose numerator and denominator are linear expressions in $ z $ with Gaussian integer coefficients.
Classically, the convergents of a finite complex continued fraction are quotients of Gaussian integers as discussed in the introduction.
In the notation of Equation \ref{aomega}, we have
$$
a_{\omega |_n} = \phi_{\omega |_n}(0),
$$
so to recover the convergents from the generalized convergents, simply set $ z=0 $.
For $ 1 \leq n \leq 3 $, we can list the generalized convergents as follows.
\begin{align}
\label{conv1}
&\frac{p_{\omega_1}(z)}{q_{\omega_1}(z)} = \frac{1}{\omega_1+z}, \\ 
&\frac{p_{\omega_1, \omega_2}(z)}{q_{\omega_1, \omega_2}(z)} = \frac{\omega_2+z}{\omega_1(\omega_2+z)+1}, \nonumber \\
&\frac{p_{\omega_1, \omega_2, \omega_3}(z)}{q_{\omega_1, \omega_2, \omega_3}(z)} = \frac{\omega_2(\omega_3+z)+1}{\omega_1(\omega_2(\omega_3+z)+1)+\omega_3+z}. \nonumber
\end{align}
For $ n \geq 4 $, the generalized convergents are generated by the recurrence relations
\begin{align}
\label{conv2}
p_{\omega_1} = 1, &\quad p_{\omega_1, \omega_2} = \omega_2, \\
&p_{\omega_1, \ldots, \omega_i} = \omega_i p_{\omega_1, \ldots, \omega_{i-1}} + p_{\omega_1, \ldots, \omega_{i-2}} \; \text{ for } \; 3 \leq i \leq n-1, \nonumber \\
&p_{\omega_1, \ldots, \omega_n}(z) = (\omega_n+z) p_{\omega_1, \ldots, \omega_{n-1}} + p_{\omega_1, \ldots, \omega_{n-2}}, \nonumber
\end{align}
and
\begin{align}
\label{conv3}
q_{\omega_1} = \omega_1, &\quad q_{\omega_1, \omega_2} = \omega_1 \omega_2 +1, \\ &q_{\omega_1, \ldots, \omega_i} = \omega_i q_{\omega_1, \ldots, \omega_{i-1}} + q_{\omega_1, \ldots, \omega_{i-2}} \; \text{ for } \; 3 \leq i \leq n-1, \nonumber \\ 
&q_{\omega_1, \ldots, \omega_n}(z) = (\omega_n+z) q_{\omega_1, \ldots, \omega_{n-1}} + q_{\omega_1, \ldots, \omega_{n-2}}. \nonumber
\end{align}
These are a straightforward modification of the classical recurrence relations for convergents of finite continued fractions (see e.g. \cite{Hen3}).

In the form of three technical lemmas, we now present some surprising duality results that emerge from these recurrence relations.

\begin{lemma}
	\label{convlem1}
	For all $ n \geq 2 $, $ \omega \in I^n $ and $ z \in \mathbb{C} $, 
	$$ 
	p_{\omega_1, \ldots, \omega_n}(z) = q_{\omega_2, \ldots, \omega_n}(z).
	$$
\end{lemma}

\begin{proof}
For $ n=2 $ and $ 3 $, this can be verified directly using Equation \ref{conv1}.
For $ n \geq 4 $, in Equations \ref{conv2} and \ref{conv3}, notice that $ p_{\omega_1, \omega_2} = q_{\omega_2} = \omega_2 $, and for $ 3 \leq i \leq n $, $ p_{\omega |_i} $ and $ q_{\omega |_i} $ are subject to the exact same recurrence relation, so the pattern persists.
\end{proof}

\begin{lemma}
	\label{convlem2}
	For all $ \omega \in I^{\ast} $, 
	$$ 
	q_{\omega}(0) = q_{\widetilde{\omega}}(0).
	$$
\end{lemma}

\begin{proof}
	Let $ \omega = (\omega_1, \ldots, \omega_n) $, and recall that we write $ q_{\omega} $ for $ q_{\omega}(0) $.
	By adapting Equation \ref{conv3}, $ q_{\widetilde{\omega}} = q_{\omega_n, \ldots, \omega_1} $ is generated by the dual recurrence relation 
	\begin{align} 
	\label{conv4}
	q_{\omega_n}=\omega_n, &\quad q_{\omega_n, \omega_{n-1}} = \omega_n \omega_{n-1} +1, \\
	&q_{\omega_n, \ldots, \omega_{n-i}} = \omega_{n-i} q_{\omega_n, \ldots, \omega_{n-i+1}}+q_{\omega_n, \ldots, \omega_{n-i+2}} \; \text{ for } \; 2 \leq i \leq n. \nonumber
	\end{align}
	We proceed by strong induction on $ n $. The case $ n=1 $ is vacuous, and $ n=2 $ follows from
	$$
	q_{\omega_1,\omega_2} = \omega_1 \omega_2+1 = q_{\omega_2, \omega_1}.
	$$
	Now assume that $ q_{\omega} = q_{\widetilde{\omega}} $ for all $ \omega $ with length $ \leq n $.
	Together with Equations \ref{conv3} and \ref{conv4}, the result follows from the following computation.
	\begin{align*}
	q_{\omega_1,\ldots,\omega_{n+1}} &= \omega_{n+1} q_{\omega_1, \ldots, \omega_n} + q_{\omega_1, \ldots, \omega_{n-1}} \\
	&= \omega_{n+1} q_{\omega_n, \ldots, \omega_1} + q_{\omega_{n-1}, \ldots, \omega_1} \\
	&= \omega_{n+1} \left( \omega_1 q_{\omega_n, \ldots, \omega_2} + q_{\omega_n, \ldots, \omega_3} \right) + \omega_1 q_{\omega_{n-1}, \ldots, \omega_2} + q_{\omega_{n-1}, \ldots, \omega_3} \\
	&= \omega_1 \left( \omega_{n+1} q_{\omega_n, \ldots, \omega_2} + q_{\omega_{n-1}, \ldots, \omega_2} \right) + \omega_{n+1} q_{\omega_n, \ldots, \omega_3} + q_{\omega_{n-1}, \ldots, \omega_3} \\
	&= \omega_1 \left( \omega_{n+1} q_{\omega_2, \ldots, \omega_n} + q_{\omega_2, \ldots, \omega_{n-1}} \right) + \omega_{n+1} q_{\omega_3, \ldots, \omega_n} + q_{\omega_3, \ldots, \omega_{n-1}} \\
	&= \omega_1 q_{\omega_2, \ldots, \omega_{n+1}} + q_{\omega_3, \ldots, \omega_{n+1}} \\
	&= \omega_1 q_{\omega_{n+1}, \ldots, \omega_2} + q_{\omega_{n+1}, \ldots, \omega_3} \\
	&= q_{\omega_{n+1}, \ldots, \omega_1}.
	\end{align*}	
\end{proof}

\begin{lemma}
	\label{convlem3}
	For all $ \omega \in I^{\ast} $, 
	$$ 
	q_{\omega}(1) = q_{\omega}(0)+p_{\widetilde{\omega}}(0).
	$$
\end{lemma}

\begin{proof}
	Let $ \omega = (\omega_1, \ldots, \omega_n) $. 
	By Equation \ref{conv3}, 
	\begin{align*}
	q_{\omega_1,\ldots,\omega_n}(1) - q_{\omega_1, \ldots, \omega_n}(0) &= (\omega_n+1)q_{\omega_1,\ldots,\omega_{n-1}} + q_{\omega_1,\ldots,\omega_{n-2}} - (\omega_n q_{\omega_1,\ldots,\omega_{n-1}} + q_{\omega_1,\ldots,\omega_{n-2}}) \\
	&= q_{\omega_1,\ldots,\omega_{n-1}}.
	\end{align*}
	Because $ q_{\omega_1, \ldots, \omega_{n-1}} = q_{\omega_1, \ldots, \omega_{n-1}}(0) $, together with Lemmas \ref{convlem1} and \ref{convlem2} we obtain
	$$
	q_{\omega_1,\ldots,\omega_n}(1) - q_{\omega_1, \ldots, \omega_n}(0) = q_{\omega_1,\ldots,\omega_{n-1}}(0) = q_{\omega_{n-1},\ldots,\omega_1}(0) = p_{\omega_n,\ldots,\omega_1}(0).
	$$
\end{proof}

\begin{lemma}
	\label{convlem4}
	For each $ \omega \in I^{\ast} $ and all $ z \in \mathbb{C} $,
	$$
	|\phi'_{\omega}(z)| = \frac{1}{|q_{\omega}(z)|^2}.
	$$
\end{lemma}

\begin{proof}
	Let $ \omega = (\omega_1, \ldots, \omega_n) $.
	By the chain rule,
	$$
	|\phi_{\omega}'(z)| = \prod_{i=1}^n |\phi'_{\omega_i} \left(\phi_{\omega_{i+1}, \ldots, \omega_n}(z)\right)|.
	$$
	Because $ |\phi'_b(z)| = |\phi_b(z)|^2 $ for each $ b \in I $, the above equation can be rewritten as
	$$
	|\phi_{\omega}'(z)| = \left(\prod_{i=1}^n |\phi_{\omega_i \ldots, \omega_n}(z)|\right)^2 = \left(\prod_{i=1}^n \frac{|p_{\omega_i, \ldots, \omega_n}(z)|}{|q_{\omega_i, \ldots, \omega_n}(z)|}\right)^2.
	$$
	By Lemma \ref{convlem1}, all the terms in this product cancel except for the numerator of the first and the denominator of the last.
	This leaves
	$$
	|\phi_{\omega}'(x)| = \left(\frac{|p_{\omega_n}(z)|}{|q_{\omega_1, \ldots, \omega_n}(z)|}\right)^2 = \frac{1}{|q_{\omega}(z)|^2},
	$$
	because $ p_{\omega_n}(z) = 1 $ from Equation \ref{conv1}.
\end{proof}

\begin{proof}[Proof of Theorem \ref{thm2}]
	For each $ \omega \in I^{\ast} $, $ q_{\omega}(z) $ is a linear expression in $ z $, with Gaussian integer coefficients.
	It is an exercise in basic complex analysis to show that a function $ D \rightarrow \mathbb{R}_{\geq 0} $ defined by $ z \mapsto |az+b| $-- where $ a,b $ are Gaussian integers-- takes its maximum and minimum at $ z=1 $ and $ z=0 $, respectively.
	Then by Lemma \ref{convlem4},
	$$
	\frac{1}{|q_{\omega}(1)|^2} \leq |\phi_{\omega}'(z)| \leq \frac{1}{|q_{\omega}(0)|^2},
	$$
	for all $ z \in D $, and the inequalities are strict unless $ z=0 $ or $ z=1 $.
	Thus
	$$
	\left(\frac{|q_{\omega}(1)|}{|q_{\omega}(0)|}\right)^{-2} \leq \frac{|\phi_{\omega}'(z)|}{|\phi_{\omega}'(w)|} \leq \left(\frac{|q_{\omega}(1)|}{|q_{\omega}(0)|}\right)^2
	$$
	for all $ z,w \in X $, with strict inequalities unless $ z=0 $ and $ w=1 $ or vice versa.
	By Lemmas \ref{convlem2} and \ref{convlem3},
	$$
	\frac{q_{\omega}(1)}{q_{\omega}(0)} = \frac{q_{\omega}(0)+p_{\widetilde{\omega}}(0)}{q_{\omega}(0)} = 1+\frac{p_{\widetilde{\omega}}(0)}{q_{\widetilde{\omega}}(0)} = 1+\phi_{\widetilde{\omega}}(0)
	$$
	which concludes the proof.
\end{proof}

\section{Bounds on Hausdorff dimension}
\label{Hbounds}
In this section, we will prove Theorem \ref{thm1} and derive some numerical estimates from it. 
In \cite{MU}, Mauldin and Urba\'{n}ski showed that for $ I = \mathbb{N} \times \mathbb{Z}i $, the continued fraction IFS $ \mathcal{S}_I $ is hereditarily regular.
Thus $ \mathcal{S}_I $ is regular for all alphabets $ I \subset \mathbb{N} \times \mathbb{Z}i $, and we may calculate $ t = \dim_H(J_I) $ by solving Bowen's equation $ P(t)=0 $.

\begin{proof}[Proof of Theorem \ref{thm1}]
By the chain rule, for each $ \omega \in I^{kn} $ we have
$$
\phi_{\omega}'(0) = \prod_{i=1}^{n-1} \phi'_{\omega_{(i-1)k+1},\ldots,\omega_{ik}} \left(\phi_{\omega_{ik+1},\ldots,\omega_{kn}}(0)\right).
$$
Applying quantitative distortion,
$$
\prod_{i=1}^n (1+\phi_{\omega_{ik},\ldots,\omega_{(i-1)k+1}}(0))^{-2} < \frac{|\phi'_{\omega}(0)|}{\prod_{i=1}^n \phi'_{\omega_{(i-1)k+1},\ldots,\omega_{ik}}(0)} < \prod_{i=1}^n (1+\phi_{\omega_{ik},\ldots,\omega_{(i-1)k+1}}(0))^2.
$$
For all $ t \geq 0 $, summing over all $ \omega \in I^{kn} $ gives
$$
\left(\sum_{\omega \in I^k} |\phi'_{\omega}(0)|^t (1+\phi_{\widetilde{\omega}}(0))^{-2t}\right)^n < \sum_{\omega \in I^{kn}} |\phi_{\omega}'(0)|^t < \left(\sum_{\omega \in I^k} |\phi'_{\omega}(0)|^t (1+\phi_{\widetilde{\omega}}(0))^{2t}\right)^n.
$$
By dividing the logarithm of this inequality by $ nk $ and taking $ n \to \infty $ we obtain $ P_k^-(t) < P(t) < P_k^+(t) $, where
$$
P_k^{\pm}(t) = \frac{1}{k} \log \left(\sum_{\omega \in I^k} |\phi'_{\omega}(0)|^t (1+\phi_{\widetilde{\omega}}(0))^{\pm 2t}\right).
$$
By Lemmas \ref{convlem2} and \ref{convlem4},
$$
|\phi'_{\omega}(0)| = \frac{1}{q_{\omega}(0)^2} = \frac{1}{q_{\widetilde{\omega}}(0)^2},
$$
so that
$$
P_k^{\pm}(t) = \frac{1}{k} \log \left(\sum_{\omega \in I^k} q_{\omega}^{-t} (1+a_{\omega})^{\pm 2t}\right)
$$
using the notation in Equations \ref{aomega} and \ref{qomega}.
For all $ k \geq 1 $ the functions $ P_k^{\pm} $ have similar properties to $ P $; they are decreasing, convex and continuous on their domains of finiteness and thus have unique zeros $ T_k^{\pm} $ satisfying equations stated in the theorem.

We now claim that $ \lim_{k \to \infty} T_k^{\pm} = \dim_H(J_I) $.
For each $ k \geq 1 $,
$$
P_k^{\pm}(t) \leq  \frac{1}{k} \log \sum_{\omega \in I^k} |\phi'_{\omega}(0)|^t \pm \frac{2t}{k} \log \sup_{\omega \in I^k} (1+a_{\omega}).
$$
Taking $ k \to \infty $ gives
\begin{equation}
\label{Ok}
|P_k^{\pm}(t) - P(t)| = O\left(\frac{1}{k}\right).
\end{equation}
Since $ T_k^{\pm} $ is the unique zero of $ P_k^{\pm} $ and $ \dim_H(J_I) $ the unique zero of $ P $, the claim follows.

To see that the convergence is monotonic, we will show that $ T_k^+ $ is monotone decreasing; the proof that $ T_k^- $ is monotone increasing is identical.
We will use the following fact from the general theory of IFS (see \cite{MU}, Theorem 4.1): $ t = \dim_H(J_I) $ satisfies
$$
\sum_{i \in I} |\phi'_i(0)|^t \leq 1.
$$
Because $ T_k^+ > \dim_H(J_I) $ for all $ k $,
$$
\sum_{i \in I} |\phi'_i(0)|^{T_k^+} < 1.
$$
From this, we have that
\begin{align*}
\sum_{\omega \in I^{k+1}} |\phi'_{\omega}(0)|^{T_k^+} (1+a_{\omega})^{2T_k^+} &= \sum_{\omega \in I^{k+1}} |\phi'_{\omega_1,\ldots, \omega_k}(\phi_{\omega_{k+1}}(0))|^{T_k^+} |\phi'_{\omega_{k+1}}(0)|^{T_k^+} (1+a_{\omega})^{2T_k^+} \\
&< \sum_{\omega_{k+1} \in I} |\phi'_{\omega_{k+1}}(0)|^{T_k^+} \sum_{\omega_1,\ldots,\omega_k \in I^k} |\phi'_{\omega_1,\ldots, \omega_k}(0)|^{T_k^+} (1+a_{\omega_1,\ldots,\omega_k})^{2T_k^+} \\
&=\sum_{i \in I} |\phi'_i(0)|^{T_k^+} < 1,
\end{align*}
so $ P_{k+1}^+(T_k^+) < 0 $.
Since $ P_{k+1}^+ $ is strictly decreasing with unique zero $ T_{k+1}^+ $, we conclude that $ T_{k+1}^+ < T_k^+ $.

It remains to estimate the speed of convergence.
By a theorem of Ruelle \cite{Rue}, $ t \mapsto P(t) $ is real analytic. 
By a Taylor expansion of $ P(t) $ in a neighborhood of its zero $ \dim_H(J_I) $,
\begin{align*}
P(t) &= \sum_{i=1}^{\infty} a_i (t-\dim_H(J_I))^i \\
&= (t-\dim_H(J_I)) \Big(a_1 + a_2(t-\dim_H(J_I))+a_3(t-\dim_H(J_I))^2 + \cdots \Big),
\end{align*}
where $ a_i $ are the Taylor coefficients. 
By Equation \ref{Ok},
\begin{align*}
|t-\dim_H(J_I)| &= \frac{|P_k^{\pm}(t)+O(\frac{1}{k})|}{|a_1 + a_2(t-\dim_H(J_I))+a_3(t-\dim_H(J_I))^2 +\cdots|} \\
&\leq \text{const.} |P_k^{\pm}(t)| + O\left(\frac{1}{k}\right)
\end{align*}
as $ k \to \infty $, for $ t $ in this neighborhood.
Because $ T_k^{\pm} \rightarrow \dim_H(J_I) $ as $ k \to \infty $, we may substitute $ t=T_k^{\pm} $ to obtain
$$
|T_k^{\pm} - \dim_H(J_I)| = O\left(\frac{1}{k}\right). 
$$
\end{proof}

\subsection{Numerical results}
In this section, we give some numerical dimension estimates that result from implementing the algorithm from Theorem \ref{thm1} in a computer algebra system.
We use Mathematica running on an ordinary PC, and limit to calculations that run in less than a minute. 
Of course, a faster machine will produce better estimates.

Given an alphabet $ I $, we wish to numerically solve the equation
\begin{equation}
\label{numsum}
\sum_{\omega \in I^k} q_{\omega}^{-t}(1+a_{\omega})^{\pm 2t} = 1
\end{equation}
for $ t $, with the largest value of $ k $ that computational speed will allow.
The values $ a_{\omega,\ldots,\omega_k} $ are just the finite continued fractions from Equation \ref{aomega}, and $ q_{\omega_1,\ldots,\omega_k} $ are the denominators of the convergents in Equation \ref{qomega}, which are generated by the recurrence relations (see \cite{Hen3})
$$
q_{\omega_1} = \omega_1, \qquad q_{\omega_1,\omega_2} = \omega_1 \omega_2+1, \qquad q_{\omega_1,\ldots,\omega_k} = \omega_k q_{\omega_1,\ldots,\omega_{k-1}} + q_{\omega_1,\ldots, \omega_{k-2}}.
$$
There are $ (\# I)^k $ summands in Equation \ref{numsum}, so we choose the largest $ k $ that our computer can manage in a short time.
For various alphabets $ I $, we chart our numerical results for $ T_k^{\pm} $, the upper and lower bounds on $ \dim_H(J_I) $, with our choice of $ k $ in each case.
For large alphabets, we take smaller $ k $ to speed up the computation.

\subsubsection{Complex continued fractions}
We begin with complex continued fractions.
Limited by computing power, we restrict to finite alphabets. 
We use the notation $ I \times Ji \subset \mathbb{N} \times \mathbb{Z}i $, for some representative choices of $ I \subset \mathbb{N} $ and $ J \subset \mathbb{Z} $.

\begin{table}[h]
	\caption{Complex continued fractions}
	\label{table1}
	\centering
	\def\arraystretch{1.2}
	\begin{tabular}{@{\extracolsep{0.5cm}} l c c c}
		\toprule
		\textbf{Alphabet} & \textbf{k}  & $ \boldsymbol{T_k^-}$ & $\boldsymbol{T_k^+}$  \\
		\midrule
		$\{2,3,4,5\} \times \{-8,\ldots,8\}i $ & 3 & 1.28512 & 1.47856 \\
		$\{2,3\} \times \{-2,\ldots,2\}i $ & 5 & 1.01264 & 1.13546 \\
		$\{3,4,5\} \times \{-8,\ldots,8\} $ & 3 & 1.13013 & 1.22647 \\
		$\{5,6,7,8\} \times \{2,3,4,5\}i $ & 4 & 0.684495 & 0.707564 \\
		$\{10,11\} \times \{10,11\}i $ & 4 & 0.255398 & 0.258506 \\
		\bottomrule\\
	\end{tabular}
\end{table}

Notice that the bounds are significantly better when the least element of $ I $ is large.

\subsubsection{Real continued fractions}
We now turn to real continued fractions, by choosing $ I \subset \mathbb{N} $.
We begin with some finite alphabets. 
Compare our Table \ref{table1} with Table 1 on pg. 16 of \cite{Hen2} and Table 3 on pg. 1442 of \cite{JP}.

\begin{table}[h]
	\caption{Real continued fractions: finite alphabets}
	\label{table1}
	\centering
	\def\arraystretch{1.2}
		\begin{tabular}{@{\extracolsep{0.5cm}} l c c c}
			\toprule
			\textbf{Alphabet} & \textbf{k}  & $ \boldsymbol{T_k^-}$ & $\boldsymbol{T_k^+}$  \\
			\midrule
			$\{1,2\}$ & 20 & 0.52417 & 0.562868 \\
			$\{2,3\}$ & 20 & 0.334398 & 0.344864 \\
			$\{5,6,7,8\}$ & 12 & 0.368563 & 0.373438 \\
			$\{10,11\}$ & 16 & 0.146668 & 0.147231 \\
			$\{100,101,102,103,104\}$ & 10 & 0.193454 & 0.193556 \\
			\bottomrule\\
		\end{tabular}
\end{table}

Again, the bounds improve when the least element of $ I $ is large.
This feature is actually shared by the Jenkinson-Pollicott algorithm \cite{JP}, but their algorithm converges exponentially fast, allowing them to calculate $ \dim_H(E_2) = 0.531280506\ldots $ to 100 decimal places, while we only manage to achieve $ 0.52417 \leq \dim_H(E_2) \leq 0.562868 $.
So for alphabets including low numbers ($ 1 $ is particularly problematic), our algorithm is woefully inferior to theirs.
Fortunately this drops off quickly; even when the lowest value is $ 2 $, the upper and lower bounds differ by $ \sim 0.01 $ for $ k=16 $.

We now turn to infinite alphabets. 
To numerically solve Equation \ref{numsum}, it is necessary to introduce a ceiling to approximate the infinite alphabet by a finite one.
For instance, for the even numbers $ I=2\mathbb{N} $, choosing the ceiling $ 1 \times 10^6 $ means that we implement the alphabet $ \{2,4,6, \ldots, 2 \times 10^6 \} $ on the computer.
Also, we can only compute first approximation $ k=1 $; this is an unfortunate restriction that a more powerful computer will not require. See Table \ref{table2}, and compare with the tables in \cite{CLU}.

\begin{table}[h]
	\caption{Real continued fractions: infinite alphabets}
	\label{table2}
	\centering
	\def\arraystretch{1.2}
	\begin{tabular}{@{\extracolsep{0.5cm}} l c c c}
		\toprule
		\textbf{Alphabet} & \textbf{Ceiling}  & $ \boldsymbol{T_1^-}$ & $\boldsymbol{T_1^+}$  \\
		\midrule
		$2\mathbb{N}$ & $ 1 \times 10^6 $ & 0.688063 & 0.856625 \\
		$3\mathbb{N}$ & $ 1 \times 10^6 $ & 0.626338 & 0.662808 \\
		$4\mathbb{N}$ & $ 1 \times 10^6 $ & 0.593185 & 0.609052 \\
		$7\mathbb{N}$ & $ 1 \times 10^6 $ & 0.544423 & 0.54838 \\
		$10\mathbb{N}$ & $ 5 \times 10^5 $ & 0.518104 & 0.519956 \\
		$100\mathbb{N}$ & $ 5 \times 10^5 $ & 0.417934 & 0.417959 \\
		\bottomrule\\
	\end{tabular}
\end{table}

As with the finite alphabets, the algorithm does substantially better when the lowest value of the alphabet is large.
Even with just the first approximation $ k=1 $ we can compute the dimension for $ I=4\mathbb{N} $ up to error $ \sim 0.01 $.

In Table \ref{table3} we list some dimension estimates for the cofinite alphabets $ F_n=\mathbb{N} \setminus \{1,\ldots,n-1\} $.
Again, we restrict to the first approximation $ k=1 $ and truncate the alphabet at the given ceiling.

\begin{table}[h]
	\caption{Real continued fractions: cofinite alphabets}
	\label{table3}
	\centering
	\def\arraystretch{1.2}
	\begin{tabular}{@{\extracolsep{0.5cm}} l c c c}
		\toprule
		\textbf{Alphabet} & \textbf{Ceiling}  & $ \boldsymbol{T_1^-}$ & $\boldsymbol{T_1^+}$  \\
		\midrule
		$F_2$ & $ 1 \times 10^6 $ & 0.791291 & 1 \\
		$F_3$ & $ 1 \times 10^6 $ & 0.759746 & 0.841966 \\
		$F_5$ & $ 1 \times 10^6 $ & 0.728387 & 0.757026 \\
		$F_{11}$ & $ 1 \times 10^6 $ & 0.692645 & 0.700367 \\
		$F_{37}$ & $ 1 \times 10^6 $ & 0.655331 & 0.656722 \\
		$F_{1000}$ & $ 1 \times 10^6 $ & 0.596801 & 0.596828 \\
		\bottomrule\\
	\end{tabular}
\end{table}

As a final note, the algorithm can be easily implemented for any given subset of $ \mathbb{N} \times \mathbb{Z}i $ with given number theoretic restrictions.
For instance, if $ I \subset \mathbb{N} $ is the powers of 2 between 16 and 1,048,576, we have $ \dim = 0.23 $ accurate to two decimal places.

\end{document}